\newfont{\Bbb}{msbm10 scaled\magstephalf}
\documentclass[12pt,reqno]{amsart}
\usepackage{amsfonts}
\usepackage{amsmath, amsthm, amscd, amsfonts}
\usepackage{amssymb}
\usepackage{amsmath}
\usepackage{xcolor}
\usepackage{amscd}
\usepackage{graphicx}
 \usepackage{epstopdf}
\allowdisplaybreaks[4]
 \newtheorem{thm}{Theorem}[section]
 \newtheorem{cor}[thm]{Corollary}
 \newtheorem{lem}[thm]{Lemma}
 \newtheorem{prop}[thm]{Proposition}
 \theoremstyle{definition}
 \newtheorem{defn}[thm]{Definition}
\theoremstyle{remark}
 \newtheorem{rem}[thm]{Remark}
 
 \numberwithin{equation}{section}
\newcommand{\pf}{\begin{proof}}
\newcommand{\zb}{\end{proof}}
\newcommand{\la}{\langle}
\newcommand{\ra}{\rangle}
\newcommand{\ol}{\overline}
\newcommand{\ma}{\mathcal}

\def\CC{{\mathbb C}}\def\DD{{\mathbb D}}

\def\lcm{\mathop{\rm lcm}\nolimits}
\def\Id{I} % we seem to use both Id and I

\begin{document}

\title[nearly invariant subspaces]{Nearly invariant subspaces for shift semigroups}
\author[Y. Liang]{Yuxia Liang}
\address{Yuxia Liang \newline School of Mathematical Sciences,
Tianjin Normal University, Tianjin 300387, P.R. China.} \email{liangyx1986@126.com}
\author[J. R. Partington]{Jonathan R. Partington}
\address{Jonathan R. Partington \newline School of Mathematics,
  University of Leeds, Leeds LS2 9JT, United Kingdom.}
 \email{J.R.Partington@leeds.ac.uk}
\begin{abstract} Let $\{T(t)\}_{t\geq0}$ be a $C_0$-semigroup
on an infinite dimensional separable Hilbert space;  a suitable definition of near $\{T(t)^*\}_{t\geq0}$ invariance of a subspace is presented in this paper. A series of prototypical examples for minimal nearly $\{S(t)^*\}_{t\geq0}$ invariant subspaces for the shift semigroup $\{S(t)\}_{t\geq0}$ on $L^2(0,\infty)$ are demonstrated, which have close links with nearly $T_{\theta}^*$ invariance on Hardy spaces of the unit disk for an inner function $\theta$. Especially, the corresponding subspaces on Hardy spaces of the right half-plane and the unit disk are related to model spaces. This work further includes a discussion on the structure of the closure of certain subspaces related to model spaces in Hardy spaces.
\end{abstract}

\maketitle

\section{Introduction}
The main aim of this paper is to investigate the near invariance problem for the shift semigroup $\{S(t)\}_{t\geq 0}$ on $L^2(0,\infty)$. In particular, we focus on characterizing a series of examples for the smallest nearly $\{S(t)^*\}_{t\geq 0}$ invariant subspaces containing certain typical functions.   The corresponding subspaces of these examples behave as model spaces in Hardy spaces of the right half-plane and the unit disc, which bring us a deeper understanding of near invariance and model spaces.\vspace{0.5mm}

Let $\ma{H}$ denote a separable infinite-dimensional Hilbert space and $\ma{L}(\ma{H})$ be the space of bounded linear operator on $\ma{H}.$ If $\{\ma{M}_i\}_{i\in I}$ is a family of subsets of the Hilbert space $\ma{H},$ we denote $\bigvee_{i\in I} \ma{M}_i$ the closed linear span generated by $\bigcup_{i\in I}\ma{M}_i$. Let the notation $\overline{\ma{M}}$ denote the closure of $\ma{M}$ for any subset $\ma{M}$ of $\ma{H}.$ Here and throughout this paper, a \emph{subspace} means a \emph{closed subspace}.\vspace{1mm}

A family $\{T(t)\}_{t\geq 0}$ in $\ma{L}(\ma{H})$ is called a $C_0$-semigroup if $T(0)=\Id,\; T(t+s)=T(t)T(s)\;\mbox{for all}\;s, t\geq 0$ and $\lim\limits_{t\rightarrow 0} T(t) x=x$ for any $x\in \ma{H}.$ Given a $C_0$-semigroup $\{T(t)\}_{t\geq 0}$ on a Hilbert space $\ma{H},$ there exists a closed and densely defined linear operator $A$ that determines the semigroup uniquely, called the generator of $\{T(t)\}_{t\geq 0},$ defined as $$Ax:=\lim\limits_{t\rightarrow 0^+} \frac{T(t)x-x}{t},$$ where the domain $D(A)$ of $A$ consists of all $x\in \ma{H}$ for which this limit exists. If $1$ is in the set $\rho(A):=\{\lambda\in \mathbb{C},\;A-\lambda I: D(A)\subset \ma{H}\rightarrow \ma{H}\;\mbox{is bijective}\},$ then $(A-I)^{-1}$ is a bounded operator on $\ma{H}$ by the closed graph theorem, and the Cayley transform of $A$ defined by
$$T:=(A+I)(A-I)^{-1}$$
is a bounded operator on $\ma{H},$ since $T-I=2(A-I)^{-1}.$ The operator $T$ is
the cogenerator and  determines the semigroup uniquely, since the generator $A$ does.  \vspace{2mm}

Let $T\in \ma{L}(\ma{H})$ be a left invertible isometric operator with finite multiplicity on $\ma{H}$. We recall
that a subspace $\ma{M}\subset \ma{H}$ is nearly $T^{-1}$ invariant if whenever $g\in \ma{H}$ and $Tg\in \ma{M},$ then $g\in \ma{M}.$
In \cite{LP2} we have shown that the nearly $T^{-1}$ invariant subspaces can be represented in terms of invariant subspaces under the backward shift. Especially, our result implies a characterization for the nearly $T_{\theta}^*$ invariant subspaces in $H^2(\mathbb{D})$ when $\theta$ is a finite Blaschke product. Here $H^2(\mathbb{D})$ is the Hardy space defined on the unit disc with the form $$H^2(\mathbb{D})=\{f:\;\mathbb{D}\rightarrow \mathbb{C}\;\mbox{analytic}, f(z)=\sum_{k=0}^\infty a_kz^k,\;\|f\|^2=\sum_{k=0}^\infty |a_k|^2<\infty\}.$$

However, there is no such simple description for the nearly $T_{\theta}^*$  invariant subspaces in $H^2(\mathbb{D})$ for an infinite Blaschke product $\theta$. In this paper, we will explore some related investigations in this direction.  Before proceeding, we recall some preliminaries appearing in many books, including \cite{P1,P2} for
 a detailed discussion.  For an infinite Blaschke product $\theta$, the Toeplitz operator $T_{\theta}^*: H^2(\mathbb{D})\rightarrow H^2(\mathbb{D})$ is universal (see, e.g. \cite{Ca}) and it is similar to the backward shift $S(1)^*$ on $L^2(0,\infty)$, given by $S(1)^*f(t)=f(t+1).$  In general, the shift semigroup $S(t):\; L^2(0,\infty) \rightarrow L^2(0,\infty)$ with $t\geq 0$ is defined by \begin{align}(S(t) f)(\zeta)=\left\{
                             \begin{array}{ll}
                               0, & \zeta\leq t, \\
                               f(\zeta-t), & \zeta>t.
                             \end{array}
                           \right.\label{shift}\end{align}
It is obvious that $S(1)^*$ is an element of the adjoint semigroup  $\{S(t)^*\}_{t\geq 0}$ given as $(S(t)^{*} f)(\zeta)=f(\zeta+t)$.

 We recall $H^2(\mathbb{C}_+)$ defined on the right half-plane $\mathbb{C}_+=\{s=x+iy,\;x>0\}$ contains all analytic functions $f: \mathbb{C}_+\rightarrow \mathbb{C}$ such that $$\|f\|_{H^2(\mathbb{C}_+)}^2=\sup_{x>0}\int_{-\infty}^\infty|f(x+iy)|^2dy<\infty.$$ The $2$-sided Laplace transform of $f\in L^1(\mathbb{R})\cap L^2(\mathbb{R})$ is given as \begin{eqnarray}(\ma{L} f)(s)=\int_{-\infty}^\infty e^{-st} f(t)dt.\label{Lap1}\end{eqnarray}
% The Paley-Wiener Theorem implies the Laplace transform provides an isomorphism of $L^2(0,\infty)$ onto $H^2(\mathbb{C}_+)$.
\begin{thm} (Paley-Wiener) The Laplace transform gives a linear isomorphism from $L^2(0,\infty)$ onto $H^2(\mathbb{C}_+)$, such that
$$\|\ma{L}(f)\|_{H^2(\mathbb{C}_+)}=\sqrt{2\pi}
\|f\|_{L^2(0,\infty)}\quad \mbox{for}\;f\in L^2(0,\infty), $$
\end{thm}
 It follows that $S(1)^*$ is unitarily equivalent to the adjoint of the multiplication operator $M_{e^{-s}}$ on the Hardy space $H^2(\mathbb{C}_+)$. Regarding the Hardy spaces $H^2(\mathbb{D})$ and $H^2(\mathbb{C}_+)$, there exists an isometric isomorphism $V:\; H^2(\mathbb{D})\rightarrow H^2(\mathbb{C}_+)$ given as
\begin{eqnarray} (Vf)(s)=\frac{1}{\sqrt{\pi}(1+s)}f(M(s)),\label{Vmap}\end{eqnarray} where $M: s\rightarrow \frac{1-s}{1+s}$ is a self-inverse bijection from  $\mathbb{C}_+$ to  $\mathbb{D}$.

Meanwhile, the inverse map $V^{-1}: H^2(\mathbb{C}_+)\rightarrow H^2(\mathbb{D})$ is defined by
\begin{eqnarray} (V^{-1}g)(z)=\frac{2\sqrt{\pi}}{1+z}g(M(z)).\label{V-map}\end{eqnarray}

Of great importance in operator-related function theory are the shift operators, ubiquitous in applications. It is well known that the inner functions arose from the representation of shift invariant subspaces in $H^2(\mathbb{D})$.
Specifically, we say $u$ is an inner function if it is a bounded analytic function on $\mathbb{D}$ such that $|u(\zeta)|=1$ for almost every $\zeta\in \mathbb{T}.$
The celebrated theorem of Beurling says that the nontrivial invariant subspaces of $H^2(\mathbb{D})$ for the forward shift operator $Sf(z)=zf(z)$ are precisely $uH^2(\mathbb{D})$ with $u$ is an inner function. At the same time,
the model space denoted by $K_u:=(uH^2(\mathbb{D}))^\perp=H^2(\mathbb{D})\ominus uH^2(\mathbb{D})$ is an invariant subspace of the backward shift operator $S^*f(z)=(f(z)-f(0))/z$; for a more detailed exposition on inner functions and model spaces see, e.g. \cite{CGP1,GMR}.

Research on invariant subspaces leads to the concept of near invariance. The study of nearly invariant subspaces for the backward shift in $H^2(\mathbb{D})$ was first explored by Hayashi \cite{Ha}, Hitt~\cite{hitt}, and then Sarason \cite{Sa1,Sa2} in relation with kernels of Toeplitz operators. Afterwards, C\^{a}mara and Partington continue the systematic investigations on near invariance and Toeplitz kernels (see, e.g. \cite{CaP1,CaP}). In particular, Hitt proved the following most widely known characterization of nearly $S^*$ invariant subspaces in $H^2(\mathbb{D})$.

\begin{thm}\cite[Proposition 3]{hitt} \label{thm Hitt}The nearly $S^*$ invariant subspaces have the form $M=uK$, with $u\in M$ of unit norm, $u(0)>0,$ and $u$ orthogonal to all elements of $M$ vanishing at the origin, $K$ is an $S^*$ invariant subspace, and the operator of multiplication by $u$ is isometric from $K$ into $H^2(\mathbb{D})$.
\end{thm}

As a nontrivial extension of our recent work in \cite{LP2}, we study the nearly invariant subspaces for the shift semigroup on $L^2(0,\infty),$ which is related to nearly $T_\theta^*$ invariance on $H^2(\mathbb{D})$ for an infinite Blaschke product $\theta.$  To the best of our knowledge, there have been no such investigations, even though there is a long history on invariant subspaces for a $C_0$-semigroup, which are defined as below.

\begin{defn}\label{defn invariant} Given a $C_0$-semigroup $\{T(t)\}_{t\geq0}$ in  $\ma{L}(\ma{H}),$ a subspace $\ma{M}\subseteq \ma{H}$ is said to be $\{T(t)\}_{t\geq0}$ invariant if $T(t)\ma{M}\subset \ma{M}$ for all $t\geq0.$\end{defn}

%Specially, for the $C_0$-semigroup $\{M(t)=e^{-st}\}_{t\geq 0}$ on $H^2(\mathbb{C}_+),$ the following theorem presents all its invariant subspaces.
%\begin{thm}\cite[Theorem 3.1.5]{P2}\label{thm jo} A non-zero subspace $\ma{K}\subset H^2(\mathbb{C}_+)$ satisfies $M(t)(\ma{K})$ $\subset \ma{K}$ for all $t\geq 0$ if and only if $\ma{K}=\phi H^2(\mathbb{C}_+)$ for some inner function $\phi\in H^\infty (\mathbb{C}_+),$ that means  $|\phi(iy)| =1$ almost everywhere for $y \in \mathbb{R}.$\end{thm}

Based on Definition \ref{defn invariant}  it might be natural to call $\ma{N}$   a nearly $\{T(t)^{*}\}_{t\geq 0}$ invariant subspace if
whenever $T(t)x \in \ma{N}$ for all $t>0,$ then $x\in \ma{N}.$
However all closed subspaces have this property since  $x=\lim\limits_{n \to \infty}T(t_n)x$ for any sequence $(t_n)$
tending to $0$, and so this definition is not useful.
We provide a more suitable definition as follows.

 \begin{defn}\label{defn nearly} Let $\{T(t)\}_{t\geq0}$ be a $C_0$-semigroup in $\ma{L}(\ma{H})$ and $\ma{N}\subseteq \ma{H}$ be a subspace. If for every $x\in \ma{H}$ whenever $T(t)x \in \ma{N}$ for some $t>0,$ then $x\in \ma{N},$ we call $\ma{N}$ a  nearly $\{T(t)^{*}\}_{t\geq 0}$ invariant subspace.\end{defn}

We say $\ma{N}$ is a trivial nearly $\{T(t)^*\}_{t\geq 0}$ invariant subspace if no element in $\ma{N}$ satisfies the above condition in Definition \ref{defn nearly}. And then we study the starting question below.\vspace{1mm}

  \emph{\textbf{\emph{Question 1}.}\;What is  the structure of nontrivial nearly $\{S(t)^*\}_{t\geq 0}$  invariant subspaces of the shift semigroup $\{S(t)\}_{t\geq 0}$ on $L^2(0,\infty)$ given in \eqref{shift}?}\vspace{1mm}

For the shift semigroup $\{S(t)\}_{t\geq 0}$ on $L^2(0,\infty)$, the maps in \eqref{V-map} and \eqref{Lap1} imply the following commutative diagrams.
 \begin{eqnarray*}\begin{CD}
L^2(0,\infty) @>S(t)>> L^2(0,\infty)\\
@VV \ma{L}  V @VV \ma{L} V\\
H^2(\mathbb{C}_+)@>M(t)>> H^2(\mathbb{C}_+)\\
@VV {V^{-1}}    V  @VV  {V^{-1}}   V\\
H^2(\mathbb{D}) @> T(t)>> H^2(\mathbb{D}).
\end{CD}\label{commute1}\end{eqnarray*}
Here the multiplication semigroup $\{M(t)\}_{t\geq 0}$ on $H^2(\mathbb{C}_+)$ is defined by
\begin{eqnarray*}(M(t) g)(s)=e^{-st} g(s),\;s\in \mathbb{C}_+,\end{eqnarray*}
 and $(M(t)^{*} g)(s)=P_{H^2(\CC_+)}e^{st} g(s).$
Moreover, $\{T(t)\}_{t\geq 0}$ on $H^2(\mathbb{D})$ is given as
\begin{eqnarray*}(T(t) h)(z)=\phi^t(z) h(z),\;z\in \mathbb{D},\end{eqnarray*}
and $(T(t)^{*} h)(z)= P_{H^2(\DD)}\phi^{-t}(z)h(z),\;z\in \mathbb{D},$ with $\phi^{t}(z):=\exp\left(-t\frac{1-z}{1+z}\right)$, the power of a standard atomic inner function.\vspace{1mm}

\begin{rem}Since every Toeplitz kernel in $H^2(\mathbb{C}_+)$ or $H^2(\mathbb{D})$ is nearly invariant under dividing by an inner function, so it is also nearly $\{M(t)^*\}_{t\geq 0}$ or $\{T(t)^*\}_{t\geq 0}$ invariant in $H^2(\mathbb{C}_+)$ or $H^2(\mathbb{D}),$ respectively.\end{rem}

It is known that the cogenerator of a $C_0$-semigroup plays an important role in invariant subspaces, and the following theorem holds.
\begin{thm}\cite[Theorem 10-9]{Fu} \label{thm cogenerator}Let $\{T(t)\}_{t\geq 0}$ be a contractive semigroup and $T$ its infinitesimal cogenerator. A subspace $\ma{M}$ is invariant under $\{T(t)\}_{t\geq 0}$ if and only if it is invariant under $T$. \end{thm}

\begin{rem} However, the parallel conclusion in Theorem \ref{thm cogenerator} does not hold for near invariance of a $C_0$-semigroup.
For example, for $\{M(t)=e^{-st}\}_{t\geq0}$ on $H^2(\mathbb{C}_+)$, $T:=(A+I)(A-I)^{-1}$ is the cogenerator of $\{M(t)\}_{t\geq0}$ with $Af:=M_{-s}f.$ Not every nearly invariant subspace in the usual sense (division by $(1-s)/(1+s)$) is nearly $\{M(t)^*\}_{t\geq 0}$ invariant, for example $e^{-s}H^2(\mathbb{C}_+)$. Likewise $((1-s)/(1+s))H^2(\mathbb{C}_+)$ is nearly $\{M(t)^*\}_{t\geq 0}$ invariant, but not nearly $T^*$ invariant.  \end{rem}

Hence it is meaningful to construct nontrivial examples of near invariance for the above well-known $C_0$-semigroups. The article is organized as follows. In Section 2, we explore a prototypical example of a smallest (cyclic) nearly $\{S(t)^*\}_{t\geq 0}$ invariant subspace in $L^2(0,\infty)$ and deduce the corresponding results for multiplication $C_0$-semigroups on Hardy spaces. The second nontrivial example is also examined in Section 3, and this leads onto a series of general  examples presented using the Hardy space model. Especially, our results reveal
that a wide class of nearly $S^*$ invariant subspaces in Hardy space $H^2(\mathbb{D})$ are of finite codimension in model spaces. The relevant characterizations in Hardy space $H^2(\mathbb{C}_+)$ are also addressed.\vspace{1mm}

In next two sections, $\ma{N}\subseteq L^2(0,\infty)$ is always supposed to be a nearly $\{S(t)^*\}_{t\geq 0}$ invariant subspace, and we denote the smallest (cyclic) nearly $\{S(t)^*\}_{t\geq 0}$ invariant subspace in $\ma{N}$ containing some nonzero vector $f$ by $[f]_{s}$. There follow two possibilities.\vspace{2mm}

(i)\; There is no function $f\in \ma{N}$, apart from the zero function, for which there exists some $\delta>0$ with $f=0$ almost everywhere on $(0,\delta)$. In this case, $\ma{N}$ is a trivial  nearly $\{S(t)^{*}\}_{t\geq 0}$ invariant subspace and $[f]_{s}=\mathbb{C}f$ for all $f \in \ma{N}$.
\vspace{1mm}

(ii)\; There are a $\delta>0$ and a function $f\in \ma{N}$ that vanishes almost everywhere on $(0,\delta)$ and not on $(0, \delta+\epsilon)$ for any $\epsilon>0.$ Since $S(\delta) S(\delta)^{*}f=f\in \ma{N}$, the  near $\{S(t)^*\}_{t\geq 0}$ invariance implies $g:=S(\delta)^{*}f\in \ma{N}$. Meanwhile, we have  $S(\lambda) g=S(\delta-\lambda)^{*} f\in \ma{N}$ for all $0\leq \lambda\leq \delta$.  So \begin{eqnarray*} [f]_{s}=\bigvee\{S(\lambda) g,\; 0\leq\lambda\leq \delta\}.\label{slambda}\end{eqnarray*}
This is the key to our work and we have not previously encountered subspaces defined in this way.

We begin with the simplest example, and explore the smallest (cyclic) nontrivial nearly $\{S(t)^*\}_{t\geq 0}$ invariant subspace in $L^2(0,\infty)$ containing $e_\delta(\zeta):=e^{-\zeta}\chi_{(\delta,\infty)}(\zeta)$ with $\delta> 0$ such that $e^{\delta}\ma{L}(e_\delta)(s)= e^{-\delta s }(1+s)^{-1}$. As an extension, we continue to take $f_{\delta,n}(\zeta):= (\zeta-\delta)^n e_\delta(\zeta)/n!$ satisfying $e^{\delta}\ma{L}(f_{\delta,n})(s)=e^{-\delta s} (1+s)^{-(n+1)}$ for integer $n\geq 1$, and formulate the Laplace transform of the smallest nearly $\{S(t)^*\}_{t\geq 0}$ invariant subspaces containing $f_{\delta,n}$ in Hardy spaces. This offers a large class of important cases for Question 1.

\section{The smallest nearly $\{S(t)^*\}_{t\geq 0}$ invariant subspace containing $e_\delta$ for $\delta>0$.}

In this section, we identify the smallest (cyclic) nearly $\{S(t)^*\}_{t\geq 0}$ invariant subspace containing $e_\delta(\zeta):=e^{-\zeta}\chi_{(\delta,\infty)}(\zeta)$
for some $\delta>0$ in $L^2(0,\infty)$. After that we express such subspaces as model spaces in Hardy spaces of the right half-plane and the unit disk.

For $\delta>0,$ let $f(\zeta)=e_\delta(\zeta)\in \ma{N}.$ For any $0\leq \lambda\leq \delta,$  $$ (S(\delta-\lambda)^{*} e_\delta)(\zeta)= e_\delta(\zeta+\delta-\lambda)
=e^{-(\delta-\lambda)}e_\lambda(\zeta),$$ and then it holds that  $$[e_\delta]_{s}:=\bigvee\{e_\lambda,\;0\leq\lambda\leq \delta\}\subseteq \ma{N}.$$

We formulate a proposition for the smallest (cyclic)  nearly $\{S(t)^{*}\}_{t\geq0}$ invariant subspace $[e_{\delta}]_{s}$ in $L^2(0,\infty).$

\begin{prop}\label{prop L2}In $L^2(0,\infty),$ the smallest nearly $\{S(t)^*\}_{t\geq 0}$ invariant subspace containing $e_\delta$ with some $\delta>0$ has the form $$[e_\delta]_{s}:=\bigvee\{e_\lambda,\;0\leq\lambda\leq \delta\}= L^2(0,\delta) +\mathbb{C}e^{-\zeta}.$$
\end{prop}

\begin{proof} To simplify the writing, we denote $N:=L^2(0,\delta) +\mathbb{C}e^{-\zeta}.$ Since $e_\lambda(\zeta)=-e^{-\zeta}\chi_{(0, \lambda]}+e^{-\zeta}\in N$ for every $0\leq \lambda\leq \delta$, so $[e_\delta]_{s}\subseteq N$.\vspace{1mm}

Conversely,  for $0\leq \lambda<\mu\leq \delta$ we have $$(e_\lambda-e_\mu)(\zeta)=e^{-\zeta}\chi_{(\lambda,\mu]}
(\zeta),$$ and next we will show that the closed linear span of  the $e_\lambda-e_\mu$ is $L^2(0,\delta).$\vspace{1mm}

Taking a function $f\in C[0,\delta],$ we will approximate $f$ arbitrarily closely in $L^\infty(0,\delta)$  by combinations of $e_{\lambda}-e_{\mu}$. Since any function $f$ can be written
as $u+iv$ with two real functions $u$ and $v$, we may suppose without loss of generality  that $f$ is real and $\|f\|_{L^\infty(0,\delta)}\leq 1$. Given $\epsilon>0$, we use uniform
continuity of $f$ to partition $[0, \delta)$ into $N$ intervals $I_k = [(k -1)\delta/N, k\delta/N)$ of length $\delta/N$ such that
\begin{eqnarray*}
\sup\limits_{I_k}f-\inf\limits_{I_k}f <\frac{\epsilon}{2}\quad \mbox{for each}\; k=1,\cdots, N.\label{Ik}\end{eqnarray*} We also choose $N$ large enough such that
\begin{eqnarray}
1-e^{-\delta/N} <\frac{\epsilon}{2}.\label{1e}
\end{eqnarray}
Then
 \begin{eqnarray}\|f-\sum_{k=1}^N a_k\chi_{I_k}\|_{L^\infty(0,\delta)}<\frac{\epsilon}{2}\label{fak}\end{eqnarray} for some suitable $a_k\in[-1, 1]$  and  that
\begin{eqnarray}&&|a_k- a_k e^{-t+(k-1)\delta/N} |
\nonumber\\&\leq&|a_k|(1-e^{-\delta/N})\nonumber
\\&\leq& 1-e^{-\delta/N}<\frac{\epsilon}{2}\quad \mbox{for}\;\;t\in I_k,\label{ake}\end{eqnarray} due to $-\delta/N<-t+(k-1)\delta/N\leq 0$ and \eqref{1e}. Then \eqref{fak} together with \eqref{ake}
give
 \begin{eqnarray*}
&&\| f-\sum_{k=1}^N a_ke^{(k-1)\delta/N}(e_{(k-1)\delta/N}-e_{k\delta/N})\|_{L^\infty(0,\delta)}
\\&&\leq \| f-\sum_{k=1}^N a_k\chi_{I_k}\|_{L^\infty(0,\delta)}+\|\sum_{k=1}^N a_k\chi_{I_k}-\sum_{k=1}^N a_ke^{-t+(k-1)\delta/N} \chi_{I_k}\|_{L^\infty(0,\delta)}
\\&&\leq \frac{\epsilon}{2}+\max\limits_{1\leq k\leq N} |a_k- a_ke^{-t+(k-1)\delta/N} | \\&&\leq \frac{\epsilon}{2}+\frac{\epsilon}{2}=\epsilon.\end{eqnarray*}
Since $C[0, \delta]$ is dense in $L^2(0,\delta)$ and $\|f\|_{L^2(0,\delta)}\leq \sqrt{\delta} \|f\|_{L^{\infty}(0,\delta)}$,  the desired result follows. \end{proof}

Using the transform $\ma{L}: L^2(0,\infty) \rightarrow H^2(\mathbb{C}_+)$ in \eqref{Lap1}, we have
\begin{eqnarray}e^{\delta}\ma{L}(e_\delta)(s)=e^\delta \int_\delta^\infty e^{-(s+1)t}dt=\frac{e^{-\delta s }}{1+s}.\label{eE}\end{eqnarray}
Then the equation in Proposition \ref{prop L2} is mapped by $\mathcal{L}$ into
\begin{equation}\bigvee\{\frac{e^{-\lambda s}}{1+s},\; 0\leq \lambda\leq \delta\}=K_{e^{-\delta s}}+ \mathbb{C} \frac{1}{1+s}.\label{Dlambda}\end{equation}
Using the map $V^{-1}:\; H^2(\mathbb{C}_+)\rightarrow H^2(\mathbb{D})$  in \eqref{V-map}, it yields that
\begin{eqnarray} V^{-1}:\; e^{-\delta s}\rightarrow \frac{2\sqrt{\pi}}{1+z}\phi^\delta(z)\quad \mbox{and}\quad V^{-1}:\; \frac{1-s}{1+s}e^{-\delta s} \rightarrow\frac{2\sqrt{\pi}}{1+z}z\phi^\delta(z).\label{Fdelta}\;\;\end{eqnarray}
Since $(1+z)^{-1}$ is an outer function in $H^2(\mathbb{D})$, \eqref{Fdelta} further implies the corresponding model spaces from $H^2(\mathbb{C}_+)$ to $H^2(\mathbb{D})$:
\begin{eqnarray*} K_{e^{-\delta s}}\rightarrow K_{\phi^\delta} \quad \mbox{and}\quad
K_{\frac{1-s}{1+s}e^{-\delta s}}\rightarrow K_{z\phi^\delta}.\end{eqnarray*}

 Next we recall a lemma for model spaces from \cite{GMR}.
%\begin{lem}\cite[Lemma 5.10]{GMR}\label{lem product} Let $\theta_1$ and $\theta_2$ be inner functions on $\mathbb{D}$. Then $K_{\theta_1\theta_2}=K_{\theta_1}\oplus \theta_1 K_{\theta_2},$ where $\oplus$ denotes the orthogonal sum. \end{lem}
\begin{lem}\cite[Corollary 5.9]{GMR}\label{lem bee} If $\theta_1$ and $\theta_2$ are inner functions on $\mathbb{D}$, then
$$K_{\theta_1}\bigvee K_{\theta_2}=K_{\lcm(\theta_1, \theta_2)}$$ where $\lcm(\theta_1, \theta_2)$  is the least common multiple of $\theta_1$ and $\theta_2.$\end{lem}

In Lemma \ref{lem bee}, if one of the left-hand subspaces is finite-dimensional, then the closed linear span is same as the sum. So it yields that \begin{eqnarray} K_{z\phi^\delta}=K_z+ K_{\phi^\delta}=\mathbb{C}+ K_{\phi^\delta}.\label{zphi}\end{eqnarray}
 Switching \eqref{zphi} into $H^2(\mathbb{C}_+)$ by the map $\eqref{Vmap}$,  we deduce \begin{eqnarray}K_{
\frac{1-s}{1+s}e^{-\delta s}}= \mathbb{C} \frac{1}{1+s}+ K_{e^{-\delta s}}.\label{decom1}\end{eqnarray}

Based on \eqref{Dlambda} and \eqref{decom1}, we obtain a corollary in $H^2(\mathbb{C}_+)$.
\begin{cor}\label{cor half}In $H^2(\mathbb{C}_+),$ the Laplace transform of $[e_\delta]_{s}$ is $$\ma{L}([e_\delta]_{s})=\bigvee\{\frac{e^{-\lambda s}}{1+s},\; 0 \leq \lambda\leq \delta\}=  K_{e^{-\delta s}} + \mathbb{C}\frac{1}{1+s}=K_{\frac{1-s}{1+s}e^{-\delta s}},$$  where $K_{e^{-\delta s}}$ and $K_{\frac{1-s}{1+s}e^{-\delta s}}$ are model spaces in $H^2(\mathbb{C}_+)$.
\end{cor}

Transferring Corollary \ref{cor half} into $H^2(\mathbb{D})$ by $$ V^{-1}:\;\frac{e^{-\lambda s}}{1+s}\rightarrow \sqrt{\pi} \phi^\lambda(z),$$ and using \eqref{zphi}, we deduce a corollary in $H^2(\mathbb{D})$.
\begin{cor}\label{cor phidelta}In $H^2(\mathbb{D}),$ it holds that $$V^{-1}(\ma{L}([e_\delta]_{s}))=\bigvee\{\phi^\lambda,\; 0\leq \lambda \leq \delta\}=K_{z \phi^{\delta}}.$$\end{cor}

%\begin{rem} Suppose the function $g$ is bounded above and below on the unit circle (e.g. an inner function or outer functions such as $z+2$), then it holds that $$\bigvee\{g \phi^\lambda,\; 0\leq \lambda\leq \delta\}=gK_{z \phi^{\delta}},$$ since multiplying by g preserves convergence and closures.
%\end{rem}

The following corollary further shows that the closed linear span of the powers of the singular inner function $\exp((z-1)/(z+1))$ is $H^2(\mathbb{D}).$

\begin{cor} \label{Cor HD}In $H^2(\mathbb{D}),$ it holds that
\begin{eqnarray} \bigvee\{\phi^\lambda,\;0\leq \lambda <\infty\}=H^2(\mathbb{D}).\label{sing1}\end{eqnarray} \end{cor}
\begin{proof} Denote $A:=\bigvee\{\phi^\lambda,\;0\leq \lambda<\infty\}.$ Suppose there is  a function $f\perp A$, then Corollary \ref{cor phidelta} implies it is orthogonal to every $$K_{z\phi^\delta}=\bigvee \{\phi^\lambda,\;0\leq \lambda \leq \delta\}.$$ This means $f$ is in the intersection of $ z\phi^\delta H^2(\mathbb{D})$ for every $\delta> 0,$ then $f=0$ by the uniqueness of inner-outer factorization.  \end{proof}
%Since the shift semigroup $\{S(t)\}_{t\geq0}$ is strongly continuous in $L^2(0,\infty)$, so it follows that $$ \|S(t) h -h\|\rightarrow 0, \;\mbox{as}\;t\rightarrow 0^+,$$ for all $h \in L^2(0,\infty)$. Using the diagram in \eqref{commute1} to transfer back to the disc, it means that $\|\phi^t f- f\|_{H^2(\mathbb{D})}\rightarrow 0$ as $t\rightarrow 0^+$ for all $f\in H^2(\mathbb{D})$. Hence it also holds \begin{eqnarray}\bigvee\{\phi^\lambda:\;0< \lambda<\infty\}=H^2(\mathbb{D}).\label{t>0} \end{eqnarray}
Using the isomorphism in \eqref{Vmap}, we have a corollary in $H^2(\mathbb{C}_+),$ which can also be deduced from the fact $1/(1+s)$ is an outer function.
\begin{cor}\label{cor 1+s} In $H^2(\mathbb{C}_+),$ it holds that
\begin{eqnarray*} \bigvee\{\frac{e^{-\lambda s}}{1+s},\;0\leq \lambda<\infty\}=H^2(\mathbb{C}_+).\label{wholeC}\end{eqnarray*} \end{cor}

Next, we require a lemma in  $H^2(\mathbb{C}_+)$.
\begin{lem} \label{lem gs}If $g(s)$, $sg(s)\in H^2(\mathbb{C}_+)$, then \begin{eqnarray*}sg(s)e^{-st}\in \bigvee\{g(s)e^{-\lambda s},\;|\lambda -t|<\epsilon\}\label{sgs}\end{eqnarray*} for all $\epsilon>0.$  \end{lem}
\begin{proof} Taking $s\in i\mathbb{R}$ and then differentiating with respect to the parameter $t$, we have that
\begin{align*} \lim\limits_{\mu\rightarrow 0} \frac{g(s)(e^{-s(t+\mu)}-e^{-st})}{\mu}=-sg(s)e^{-st}.\end{align*}
Moreover, by the mean value inequality, we deduce that
\begin{align*} \left| \frac{g(s)(e^{-s(t+\mu)}-e^{-st})}{\mu}\right|\leq \sup\limits_{|\lambda -t|<\mu}|-sg(s)e^{-\lambda s}|=|sg(s)|,\end{align*}
so that the above convergence is not only pointwise, but, using dominated convergence and the assumption that $|sg(s)|\in L^2(i\mathbb{R}),$ takes place in $H^2(\mathbb{C}_+)$ norm. Thus
% the convergence is uniformly on $\mathbb{C}_+,$ and then
the desired result follows.
\end{proof}

\begin{rem}$(1)$ For the case $g(s)=1/(1+s)^2,$ we conclude that
\begin{align*}\frac{se^{-st}}{(1+s)^2}\in B:=\bigvee\{\frac{e^{-\lambda s}}{(1+s)^2},\; 0\leq \lambda \leq \delta\}\;\;\mbox{for all}\;t\in[0,\delta]. \end{align*} Here we use one-sided limits for $t=0$ and $t=\delta$ in the Lemma \ref{lem gs}. By linearity it follows $e^{-st}/(1+s)\in B$ and further implies the inclusion \begin{align*}\bigvee\{\frac{e^{-\lambda s}}{1+s},\;0\leq \lambda\leq \delta\}\subseteq \bigvee\{\frac{e^{-\lambda s}}{(1+s)^2},\;0\leq \lambda\leq \delta\}. \end{align*} Transferring to the unit disc by the map \eqref{V-map}, it follows that
\begin{eqnarray}\bigvee\{\phi^\lambda,\;0\leq \lambda\leq \delta\}\subseteq \bigvee\{(1+z)\phi^\lambda,\;0\leq \lambda\leq \delta\}.\label{deltaD} \end{eqnarray}

$(2)$ For the general $g_n(s)=1/(1+s)^{n+1}\in H^2(\mathbb{C}_+)$ with integer $n\geq 1,$ it similarly holds on $H^2(\mathbb{C}_+)$ and $H^2(\mathbb{D})$ as below
\begin{eqnarray}&&\bigvee\{\frac{e^{-\lambda s}}{(1+s)^{n}},\;0\leq \lambda\leq \delta\}\subseteq \bigvee\{\frac{e^{-\lambda s}}{(1+s)^{n+1}},\;0\leq \lambda\leq \delta \}, \nonumber
\\&& \bigvee\{(1+z)^{n-1}\phi^\lambda,\;0\leq \lambda\leq \delta\}\subseteq \bigvee\{(1+z)^n\phi^\lambda,\;0\leq \lambda\leq \delta\}.\quad \quad\quad\label{nD} \end{eqnarray}
\end{rem}

\section{The smallest nearly $\{S(t)^*\}_{t\geq0}$ invariant subspaces in more general situations}

Recall that in Section 2, the function $e_\delta$ satisfies \eqref{eE}. Next, it is natural to look at the smallest (cyclic) nearly $\{S(t)^*\}_{t\geq 0}$ invariant subspace in $L^2(0,\infty)$ containing    $f=f_{\delta,1}(\zeta):=(\zeta-\delta) e_\delta(\zeta)$ such that $e^{\delta} \ma{L}(f_{\delta,1})(s)= e^{-\delta s} (1+s)^{-2}$ for some $\delta>0$. In this case, we
shall show that the mapped subspace in $H^2(\mathbb{D})$ is the closure of $(1+z)K_{z\phi^\delta}$ equalling the model space $K_{z^2\phi^\delta}.$ Afterwards, we describe the general formulas for the Laplace transform of the smallest (cyclic) nearly $\{S(t)^*\}_{t\geq 0}$ invariant subspace containing $f_{\delta,n}$ in $H^2(\mathbb{C}_+)$ and corresponding subspaces in $H^2(\mathbb{D})$.  This leads us to find some important characterizations for the closure of $gK_{z\phi^\delta}$ with a more general $g\in L^\infty(\mathbb{T})$. Meanwhile, we also summarize the descriptions in $H^2(\mathbb{C}_+)$.
\subsection{The smallest (cyclic) nearly $\{S(t)^*\}_{t\geq0}$ invariant subspace containing $f_{\delta,1}$ for some $\delta>0$}For the vector $f_{\delta,1}(\zeta):=(\zeta-\delta)e_\delta$, it holds that
\begin{eqnarray*} e^{\delta}\ma{L}(f_{\delta,1})(s)=\frac{e^{-\delta s}}{(1+s)^2}.\label{Lfdelta}\end{eqnarray*}

Suppose the  nearly $\{S(t)^*\}_{t\geq0}$ invariant subspace $\ma{N}$  contains $f_{\delta,1}$ with some $\delta>0.$  Since
\begin{eqnarray*}S(\delta-\lambda)^* f_{\delta,1}=e^{-(\delta-\lambda)}(\zeta-\lambda)e_\lambda\in \ma{N}\;\;\mbox{for all}\;0\leq \lambda \leq \delta,\end{eqnarray*}  the smallest (cyclic) nearly $\{S(t)^*\}_{t\geq0}$ invariant subspace containing the vector $f_{\delta,1}$ in $\ma{N}$ is
\begin{align*} [f_{\delta,1}]_{s}=\bigvee\{(\zeta-\lambda)e_\lambda,\;0\leq \lambda\leq \delta \}. \end{align*} In $H^2(\mathbb{C}_+),$ the Laplace transform maps the subspace $[f_{\delta,1}]_{s}$  onto \begin{align*}\ma{L}([f_{\delta,1}]_{s})=\bigvee\{\frac{e^{-\lambda s}}{(1+s)^2},\;0\leq \lambda\leq \delta\}.\label{co1}\end{align*}  Meanwhile, in $H^2(\mathbb{D})$, by the map $V^{-1}$ in \eqref{V-map}, we have \begin{eqnarray}
 V^{-1}(\ma{L}([f_{\delta,1}]_{s}))
&=&\bigvee\{(1+z)\phi^\lambda,\;0\leq \lambda\leq \delta\}\nonumber\\&=&\overline{(1+z)K_{z\phi^\delta}}
\subseteq \overline{K_{z\phi^\delta}+zK_{z\phi^\delta}}.
\label{D1}\end{eqnarray}
By the formula \eqref{deltaD}, it follows that $$\phi^\lambda, z\phi^\lambda \in \overline{(1+z)K_{z\phi^\delta}}$$ for $0\leq \lambda \leq \delta$. This further implies $$\overline{K_{z\phi^\delta}+zK_{z\phi^\delta}}\subseteq \overline{(1+z)K_{z\phi^\delta}},$$ which together with \eqref{D1} imply \begin{equation}V^{-1}(\ma{L}([f_{\delta,1}]_{s}))
=\overline{K_{z\phi^\delta}+zK_{z\phi^\delta}}.\label{fdeltaD} \end{equation}

The next proposition gives the concrete form of \eqref{fdeltaD}.

\begin{prop} \label{prop A1delta}In $H^2(\mathbb{D}),$ it holds that
\begin{eqnarray}V^{-1}(\ma{L}([f_{\delta,1}]_{s}))=\bigvee\{(1+z)\phi^\lambda,\;0\leq \lambda\leq \delta\}=K_{z^2\phi^\delta}.\label{A1delta} \;\;\end{eqnarray}
\end{prop}
\begin{proof}
For any $f\perp V^{-1}(\ma{L}([f_{\delta,1}]_{s})),$ \eqref{fdeltaD} implies $f\perp K_{z\phi^\delta}$ and $f\perp zK_{z\phi^\delta},$ which means $f\in z\phi^\delta H^2(\mathbb{D})$ and $S^*f\in z\phi^\delta H^2(\mathbb{D}).$ So we can suppose $f=z\phi^\delta h$ with some $h\in H^2(\mathbb{D}),$ and then $S^*f=\phi^\delta h\in z\phi^\delta H^2(\mathbb{D}),$ verifying $z$ divides $h$. Hence $f\in z^2\phi^\delta H^2(\mathbb{D}),$ this shows $$K_{z^2\phi^\delta}\subseteq  V^{-1}(\ma{L}([f_{\delta,1}]_{s})).$$ Further,  since $K_{z\phi^\delta}\subseteq K_{z^2\phi^\delta}$ and $zK_{z\phi^\delta}\subseteq K_{z^2\phi^\delta},$ so combining with \eqref{fdeltaD} we obtain \eqref{A1delta}.
\end{proof}
It is well known that a continuous operator $T$ acting between Banach spaces $X$ and $Y$ is bounded below if and only if $T$ is injective and has closed
range. Using this, we can show the fact below.

\begin{prop} The subspace $(1+z)K_{z\phi^\delta}$ is not closed in $H^2(\mathbb{D})$. \end{prop}
\begin{proof}
Let $$k_v(z)=\frac{1-\overline{v}\overline{\phi^\delta(v)} z\phi^\delta(z)}{1-\overline{v}z}$$ denote the reproducing kernel for the model space $K_{z\phi^\delta}.$ Then
\begin{eqnarray*}\|k_v\|^2=k_v(v)=\frac{1-|v\phi^\delta(v)|^2}
{1-|v|^2}. \end{eqnarray*} For $v\rightarrow -1$ nontangentially, it holds that  $\phi^\delta(v)\rightarrow 0$ and then \begin{eqnarray}\|k_v\|^2\rightarrow \infty .\label{kvinfty}\end{eqnarray}

On the other hand, it holds $\|(1+z)k_v\|^2=2\|k_v\|^2+ 2\mbox{Re} \la k_v, zk_v\ra.$ Since $k_v$ is orthogonal to $\overline{v}\overline{\phi^\delta(v)}z^2\phi^\delta(z)
(1-\overline{v}z)^{-1}$, we have that
\begin{eqnarray*} \la k_v,zk_v\ra&=&\la \frac{1-\overline{v}\overline{\phi^\delta(v)} z\phi^\delta(z)}{1-\overline{v}z},\frac{z}{1-\overline{v}z}\ra\\&=&
\la \frac{1}{1-\overline{v}z}, \frac{z}{1-\overline{v}z}\ra-\la \frac{\overline{v}
\overline{\phi^\delta(v)}z\phi^\delta(z)}{1-\overline{v}z}, \frac{z}{1-\overline{v}z}\ra\\&=& \frac{\overline{v}(1-|\phi^\delta(v)|^2)}{1-|v|^2},\end{eqnarray*}
where we use the fact $(1-\overline{v}z)^{-1}$ is the reproducing kernel for $H^2(\mathbb{D}).$
%\begin{eqnarray*}&&\la k_v, \frac{\overline{v}\overline{\phi^\delta(v)}z^2\phi^\delta(z)}
%{1-\overline{v}z}\ra\\&&= \la \frac{1}{1-\overline{v}z}, \frac{\overline{v}\overline{\phi^\delta(v)}z^2\phi^\delta(z)}
%{1-\overline{v}z}\ra -\la \frac{\overline{v}\overline{\phi^\delta(v)}z\phi^\delta(z)}
%{1-\overline{v}z}, \frac{\overline{v}\overline{\phi^\delta(v)}z^2\phi^\delta(z)}
%{1-\overline{v}z}\ra\\&&= \frac{|v\phi^\delta(v)|^2\overline{v}}{1-|v|^2}-
%\la \frac{1}{1-\overline{v}z}, \frac{|v|^2|\phi^\delta v|^2z}{1-\overline{v}z}\ra=0.
%\end{eqnarray*}
%This gives
%\begin{eqnarray*}\la k_v, \frac{z}{1-\overline{v}z}\ra \end{eqnarray*}
 So we deduce that
\begin{eqnarray*}\|(1+z)k_v\|^2 =2\frac{1-|v\phi^\delta(v)|^2}{1-|v|^2}+2\mbox{Re} \left( \frac{\overline{v}(1-|\phi^\delta(v)|^2)}{1-|v|^2}\right). \end{eqnarray*}
Let $v=-r$ with $0<r<1,$ and then we get that
\begin{eqnarray}\|(1+z)k_v\|^2&=&
2\frac{1-|r\phi^\delta(-r)|^2}{1-r^2}-2r\frac{1-|\phi^\delta(-r)|^2}{1-r^2}
\nonumber\\&=& \frac{2-2r}{1-r^2}-\frac{(2r^2-2r)|\phi^\delta(-r)|^2}{1-r^2}\nonumber\\
 &=& \frac{2}{1+r}+2r\frac{|\phi^\delta(-r)|^2}{1+r}\nonumber
\\&\rightarrow& 1\;\mbox{as}\quad r\rightarrow 1.\label{z+1kv}
 \end{eqnarray}
Then \eqref{kvinfty} together with \eqref{z+1kv} imply the injective map $f\rightarrow (1+z)f$ is not bounded below on $K_{z\phi^\delta},$ so  $(1+z)K_{z\phi^\delta}$ is not closed in $H^2(\mathbb{D})$. \end{proof}
In $H^2(\mathbb{C}_+),$ Proposition \ref{prop A1delta} implies a characterization for the subspace $\ma{L}([f_{\delta,1}]_{s})$ with some $\delta>0$.
\begin{thm}\label{thm Lap} The Laplace transform of the smallest nearly $\{S(t)^*\}_{t\geq 0}$ invariant subspace containing $f_{\delta,1}$ with some $\delta>0$ has the form \begin{align*}\ma{L}([f_{\delta,1}]_{s})=\bigvee\{\frac{e^{-\lambda s}}{(1+s)^2},\;0\leq \lambda\leq \delta\}=
K_{\left(\frac{1-s}{1+s}\right)^2e^{-\delta s}},\end{align*}
where $K_{\left(\frac{1-s}{1+s}\right)^2e^{-\delta s}}$ is a model space in $H^2(\mathbb{C}_+)$.
\end{thm}

\subsection{More general examples on the smallest (cyclic) nearly invariant subspaces} In this subsection, we suppose the nearly $\{S(t)^*\}_{t\geq 0}$ invariant subspace $\mathcal{N}\subset L^2(0, \infty)$ contains the function $f_{\delta,n}$ in Lemma \ref{lem fndelta}, which can degenerate $e_\delta$ and  $f_{\delta,1}$ with $n=0$ and $n=1.$

\begin{lem} \label{lem fndelta} Using the Laplace transform in \eqref{Lap1}, we have  \begin{eqnarray}e^{\delta} \ma{L}(f_{\delta,n})(s)=\frac{e^{-\delta s}}{(1+s)^{n+1}}\label{Lfdeltan}\end{eqnarray} holds for the functions  \begin{eqnarray}f_{\delta,n}(\zeta)=\frac{(\zeta-\delta)^n}{n!}
e_\delta(\zeta) \label{fdeltan} \end{eqnarray} with $e_{\delta}(\zeta)=e^{-\zeta}\chi_{(\delta,\infty)}(\zeta),$ $\delta>0,$ and any nonnegative integer $n$.
\end{lem}
\begin{proof}
Denote $$I_n(s)=\int_\delta^\infty e^{-(s+1)t}(t-\delta)^n dt,$$ it follows that
\begin{eqnarray*} I_n(s)=\frac{n}{1+s}I_{n-1}(s).\label{Inn-1} \end{eqnarray*}
By iterations, we further have
\begin{eqnarray*}I_n(s)=\frac{n!}{(1+s)^{n}}I_0(s)
=\frac{n!e^{-(s+1)\delta}}{(1+s)^{n+1}},\label{In} \end{eqnarray*} by the display \eqref{eE} for $I_0(s)$. And then it turns out
$$\ma{L}(f_{\delta,n})(s)=\frac{I_n(s)}{n!}
=\frac{e^{-(s+1)\delta}}{(1+s)^{n+1}}.$$ This means the equation  \eqref{Lfdeltan} is true.
\end{proof}
 Here we first present the mapped subspaces of the smallest (cyclic) nearly $\{S(t)^*\}_{t\geq 0}$ invariant subspace $[f_{\delta,n}]_s$  in Hardy spaces.
\begin{thm}\label{thm n} For any nonnegative integer $n$ and $\delta>0$, the following statements are true.

$(1)$ In $H^2(\mathbb{D}),$ it holds that the cyclic nearly $\{T(t)^*\}_{t\geq 0}$ invariant subspace $\bigvee\{(1+z)^n\phi^\lambda,\; 0\leq \lambda\leq \delta\}=K_{z^{n+1}\phi^\delta};$\vspace{2mm}

$(2)$ In $H^2(\mathbb{C}_+),$ it holds that the cyclic nearly $\{M(t)^*\}_{t\geq 0}$ invariant subspace $
 \bigvee\{\frac{e^{-\lambda s}}{(1+s)^{n+1}},\; 0\leq \lambda\leq \delta\}=K_{ \left(\frac{1-s}{1+s}\right)^{n+1}e^{-\delta s}}.$
\end{thm}
\begin{proof} Since $(2)$ can be deduced using the map $V$ in \eqref{V-map} and the result $(1),$ we only need to prove $(1).$ By
mathematical induction, it holds for $n=0,$ and we suppose it is true for $n-1$, that is, $\bigvee\{(1+z)^{n-1}\phi^\lambda,\; 0\leq \lambda\leq \delta\}=K_{z^{n}\phi^\delta}$. Then it turns out that
\begin{eqnarray*}
 &&\bigvee\{(1+z)^n\phi^\lambda,\; 0\leq \lambda\leq \delta\}\\&&=\overline{(1+z)\bigvee\{(1+z)^{n-1}\phi^\lambda,\; 0\leq \lambda\leq \delta\}}\\&&=\overline{(1+z) K_{z^n\phi^\delta}}.\end{eqnarray*}
 By the formula in \eqref{nD}, it follows that $(1+z)^{n-1}\phi^\lambda,\;z(1+z)^{n-1}\phi^\lambda \in \overline{(1+z) K_{z^n\phi^\delta}}$ for $0\leq \lambda \leq \delta$, implying
 $$\overline{K_{z^n\phi^\delta}+zK_{z^n\phi^\delta}}\subseteq \overline{(1+z) K_{z^n\phi^\delta}}.$$ Since the above converse inclusion is obvious, it yields that $$
 \bigvee\{(1+z)^n\phi^\lambda,\; 0\leq \lambda\leq \delta\}=\overline{K_{z^n\phi^\delta}+zK_{z^n\phi^\delta}}.$$ By the similar proof of Proposition \ref{prop A1delta}, we obtain $$\overline{K_{z^n\phi^\delta}+zK_{z^n\phi^\delta}}=K_{z^{n+1}\phi^\delta}.$$
\end{proof}
Now we can formulate a corollary for the Laplace transform of  $[f_{\delta,n}]_s.$
\begin{cor} The Laplace transform of the smallest nearly $\{S(t)^*\}_{t\geq 0}$ invariant subspace containing $f_{\delta,n}$ in \eqref{fdeltan} has the form
 $$\ma{L}([f_{\delta,n}]_s]=K_{ \left(\frac{1-s}{1+s}\right)^{n+1}e^{-\delta s}}$$ for $\delta>0$ and any nonnegative integer $n$.
 \end{cor}
 \begin{rem} It is known that $\{\sqrt{2\pi}p_n(t)e^{-t}\}_{n=0}^\infty$ forms an orthonormal basis for $L^2(0,\infty),$ with $p_n(t)=\pm L_n(2t)/\sqrt{\pi}$ (a real polynomial of degree $n$) and $L_n$ denotes the Laguerre polynomial $L_n(t)=\frac{e^t}{n!} \frac{d^n}{dt^n}(t^ne^{-t})$. So when we consider the smallest (cyclic) nearly $\{S(t)^*\}_{t\geq 0}$ invariant subspace containing $f_{\delta,n}$ in \eqref{fdeltan} for some $\delta>0$ and nonnegative integer $n,$ it covers many important cases for the Question 1. \end{rem}

\subsection{The characterization of $\overline{gK_{z\phi^\delta}}$ for some general $g$}
Inspired by the result $(1)$ in Theorem \ref{thm n}, we continue to present the concrete formula of the subspace $$c(g):=\overline{gK_{z\phi^\delta}}=\bigvee\{g\phi^\lambda,\;0\leq \lambda\leq \delta\}$$ for a more general function $g\in L^\infty(\mathbb{T}).$
First of all, we demonstrate the subspace like $\overline{gK_{\theta}}$ is nearly $S^*$ invariant for $g\in H^\infty(\mathbb{D})$ with $g(0)\neq 0$
(not necessarily an isometric multiplier)
 and a non-constant inner function $\theta.$

\begin{thm} \label{thm nearlyS*} Let $g\in H^\infty(\mathbb{D})$ with $g(0)\neq 0$ and $\theta$ a non-constant inner function. Then $\overline{g K_\theta}$ is nearly $S^*$ invariant, and so by Hitt's theorem it can be written as $hK,$ where $K$ is $S^*$-invariant (either a model space or $H^2(\mathbb{D})$ itself) and $h\in H^2(\mathbb{D})$ is a function such that multiplication by $h$ is isometric on $K$.\end{thm}

\begin{proof}Since $\theta$ is non-constant, there is an $n_0\geq 1$ such that $(S^{*n_0} \theta)(0)\neq 0.$ Without loss of generality, suppose  $g(0)(S^{*n_0} \theta)(0)=1.$   Take $f\in \overline{gK_\theta}$ with $f(0)=0,$ so there exist $f_n\in gK_\theta$ with $\|f_n-f\|_2\rightarrow 0$ as $n\rightarrow \infty.$ Particularly, $f_n(0)=f_n(0)-f(0)=\la f_n-f, 1\ra\rightarrow 0,$ as $n\rightarrow \infty.$\vspace{1mm}

Let $F_n=f_n-f_n(0) g S^{*n_0} \theta\in gK_\theta$ (since $S{^{* n_0}} \theta \in  K_\theta$)  with $F_n(0)=0$ and $F_n\rightarrow f$ as $n\rightarrow\infty$. Further it holds $F_n=gk_n$ with $k_n\in K_\theta$ such that $k_n(0)=0$ and $S^* F_n=F_n/z=gk_n/z=gS^* k_n\in gK_\theta$ due to $K_\theta$ is $S^*$ invariant. Now we conclude that
\begin{align*} \|S^* F_n-S^* f\|_2&\leq   \|S^*\|\cdot\|F_n-f\|_2\rightarrow 0,\end{align*}as $n\rightarrow \infty.$ This means $S^* f\in \overline{gK_\theta},$ ending the proof.\end{proof}

%\begin{lem} Let $g\in H^\infty$ with $g(0)\neq 0$ and $\theta$ an inner function such that $\theta(0)=0.$ Then $\overline{g K_\theta}$ is nearly invariant in the usual sense, and so by Hitt's theorem can be written as $hK,$ where $K$ is $S^*$-invariant (either a model space or $H^2(\mathbb{D})$ itself) and $h\in H^2(\mathbb{D})$ is a function such that multiplication by $h$ is isometric on $K$.\end{lem}
%\begin{proof}Without loss of generality $g(0)=1.$ Take $f\in \overline{gK_\theta}$ with $f(0)=0,$ so there exist $f_n\in gK_\theta$ with $\|f_n-f\|_2\rightarrow 0$ as $n\rightarrow \infty.$ In particular $f_n(0)=f_n(0)-f(0)=\la f_n-f, 1\ra\rightarrow 0,$ as $n\rightarrow \infty.$

%Let $F_n=f_n-f_n(0) g S^* \theta\in gK_\theta$ due to $1\in K_\theta,$ and then $F_n(0)=0$. Further it holds $F_n=gk_n$ with $k_n\in K_\theta$ such that $k_n(0)=0,$ and $S^* F_n=F_n/z=gk_n/z=gS^* k_n\in gK_\theta$ due to $K_\theta$ is $S^*$ invariant. Now we obtain
%\begin{align*} \|S^* F_n-S^* f\|_2&\leq & \|S^*\|\cdot\|F_n-f\|\\&\leq&\|S^*\|\cdot(\| f_n-f\|+||f_n(0)|\|g\|_\infty)\rightarrow 0,\end{align*}as $n\rightarrow \infty.$ This means $S^* F_n\rightarrow S^* f,$ so $S^* f\in \overline{gK_\theta},$ ending the proof.\end{proof}

It is known that every rational function $p/q\in H^2(\mathbb{D})$ in its lowest terms has $q$ invertible in $H^\infty(\mathbb{D})$, and without loss of generality $p$ is a polynomial with zeros on the unit circle $\mathbb{T}$, since zeros inside the open disc can be removed using Blaschke factors. Next we concentrate on finding $c(\widetilde{p})$ when $\widetilde{p}$ is a polynomial with zeros on $\mathbb{T}$.

Denote $\widetilde{p}_N(z):=\prod_{j=1}^N (z+w_j)$, $N\geq 1$ and $w_j\in \mathbb{T}$ for $j=1,\cdots, N.$ Theorem \ref{thm nearlyS*} implies $c(\widetilde{p}_N)$ is a nearly $S^*$ invariant subspace in $H^2(\mathbb{D}).$ For a further description of $c(\widetilde{p}_N),$ we cite a result from \cite{CaP2}. We say $h\in H^2(\mathbb{D})\setminus\{0\}$ is contained in a \emph{minimal} Toeplitz kernel $\mathcal{K}_{min}(h)$ means that every Toeplitz kernel $K$ with $h\in K$ contains $\mathcal{K}_{min}(h).$
\begin{lem}\label{lem Toeplitz kernel}  \cite[Theorem 3.3]{CaP2} Let $h\in H^2\setminus\{0\}$ and $h=IO$ be its inner-outer factorization. Then there exists a minimal Toeplitz kernel containing span$\{h\},$ written $\mathcal{K}_{min}(h)$ with $$ \mathcal{K}_{min}(h)=ker T_{\overline{z}\ol{IO}/O}.$$\end{lem}

%$(2)$ \cite[Theorem 4.1]{CaP2} Let $g\in L^\infty\setminus\{0\}$ be such that $ker T_g$ is nontrivial. Then $k$ is a maximal vector for $Ker T_g$ if and only if $k\in H^2$ and  $k=g^{-1} \overline{zp}$, where $p\in H^2$ is outer.

Now Lemma \ref{lem Toeplitz kernel} implies $c(\widetilde{p}_N)\subseteq\ma{K}_{min}(\widetilde{p}_N\phi^\delta),$  since, being a Toeplitz kernel, it will contain $\widetilde{p}_N\phi^{\lambda}$ for all $0\leq \lambda\leq \delta.$ And it yields that
 $c(\widetilde{p}_N)\subseteq ker T_d$ with the function $d\in L^\infty(\mathbb{T})$ and
$$d(z):=\frac{\overline{z\phi^\delta(z) \widetilde{p}_N(z)}}{\widetilde{p}_N(z)}=\overline{z^{N+1}\phi^\delta}(z) (\prod_{j=1}^N\overline{w_j}). $$ So we conclude that $c(\widetilde{p}_N)\subseteq  ker T_{\overline{z^{N+1}\phi^\delta}}=K_{z^{N+1}\phi^\delta}.$
Next we explore the gap between  $c(\widetilde{p}_N)$ and $K_{z^{N+1}\phi^\delta}.$

\begin{prop} \label{prop poly}Let $\widetilde{p}_N(z):=\prod_{j=1}^N (z+w_j)$ with $w_j\in \mathbb{T}$,   $j=1,\cdots, N,$ it follows that \begin{equation} c(\widetilde{p}_N)+ \phi^\delta K_{z^N}
=K_{z^{N+1}\phi^\delta}.\label{pN}\end{equation} Hence $c(\widetilde{p}_N)$ has codimension
at most $N$ in $K_{z^{N+1}\phi^\delta}.$  \end{prop}
\begin{proof} Since $\widetilde{p}_N(z):=\prod_{j=1}^N (z+w_j)$ is an outer function in $H^2(\mathbb{D}),$ by \eqref{sing1}, we obtain that
$$\bigvee\{\widetilde{p}_N\phi^\lambda,\;0\leq \lambda<\infty\}=H^2(\mathbb{D}).$$  Denote $B_N:=\bigvee\{\widetilde{p}_N\phi^\lambda,\;\delta\leq \lambda<\infty\}$ and use \eqref{sing1} again to obtain
\begin{align*}B_N&=\phi^\delta\bigvee\{\widetilde{p}_N\phi^\lambda,\;0\leq \lambda<\infty\}\\&=\phi^\delta H^2(\mathbb{D})\\&= \phi^\delta(z^{N+1} H^2(\mathbb{D})\oplus  K_{z^{N+1}}) \\&=z^{N+1}\phi^\delta H^2(\mathbb{D}) \oplus(\mathbb{C}\widetilde{p}_N\phi^\delta+ \bigvee\{z^k\phi^\delta,\; 0\leq k\leq N-1\})\\&= z^{N+1}\phi^\delta H^2(\mathbb{D}) \oplus(\mathbb{C}\widetilde{p}_N\phi^\delta+ \phi^\delta K_{z^N}).\end{align*}

Since $\overline{c(\widetilde{p}_N)+ B_N}=H^2(\mathbb{D})$ and $\widetilde{p}_N\phi^\delta \in c(\widetilde{p}_N),$
it always holds that
\begin{align*}\left\{
                   \begin{array}{ll}
                     &\overline{(c(\widetilde{p}_N)+\phi^\delta K_{z^N})+z^{N+1}\phi^\delta H^2(\mathbb{D})} =H^2(\mathbb{D}), \vspace{1mm}\\
                    &(c(\widetilde{p}_N)+ \phi^\delta K_{z^{N}})\perp z^{N+1}\phi^\delta H^2(\mathbb{D}).
                   \end{array}
                 \right.
                 \end{align*}
 This means
$$\overline{(c(\widetilde{p}_N)+\phi^\delta K_{z^N})\oplus z^{N+1}\phi^\delta H^2(\mathbb{D})} =H^2(\mathbb{D}),$$ which is equivalent to saying

$$(c(\widetilde{p}_N)+\phi^\delta K_{z^N})\oplus z^{N+1}\phi^\delta H^2(\mathbb{D}) =H^2(\mathbb{D}).$$ This further yields the desired result.
\end{proof}

For a more $g\in L^\infty(\mathbb{T}),$ we  deduce the following theorem on the closure of $gK_{z\phi^\delta}$.
\begin{thm}\label{thm generalg} Suppose $g(z)=\widetilde{p}_N(z)h(z)$ with $\widetilde{p}_N(z):=\prod_{j=1}^N (z+w_j)$, where $w_j\in \mathbb{T}$,  $j=1,\cdots, N,$ and $h$ is an invertible rational function in $L^\infty(\mathbb{T})$. Then $c(g)$ has codimension at most $N$ in $hK_{z^{N+1}\phi^\delta},$ that is,
\begin{equation}c(g)+h\phi^\delta K_{z^N}=hK_{z^{N+1}\phi^\delta}. \label{gN}\end{equation} \end{thm}
\begin{proof}
By the fact $h$ is invertible in $L^\infty(\mathbb{T})$, we can multiply the equation \eqref{pN} by $h$  to deduce \eqref{gN}. \end{proof}

\begin{rem} In Theorem \ref{thm generalg}, letting $g(z)=(1+z)^N h(z)$ with an invertible rational function $h\in L^\infty(\mathbb{T}),$ we deduce $$\bigvee\{(1+z)^N h \phi^\lambda,\;0\leq \lambda\leq \delta\}=hK_{z^{N+1}\phi^\delta},$$ which has codimension $0$ in $hK_{z^{N+1}\phi^\delta}.$ Particularly, for $h(z)=1,$  the result $(1)$ of Theorem \ref{thm n} implies $$\bigvee\{(1+z)^N \phi^\lambda,\;0\leq \lambda\leq \delta\}=K_{z^{N+1}\phi^\delta},$$ which has codimension $0$ in $K_{z^{N+1}\phi^\delta}.$
\end{rem}

In the sequel, we apply Theorem \ref{thm generalg} to describe the corresponding case for rational functions in $H^2(\mathbb{C}_+)$. Note that a rational function $g$ in $H^2(\mathbb{C}_+)$ can be factorized as $g=g_ig_o,$ where $g_i$ is inner and hence invertible in $L^\infty(i\mathbb{R})$ (it is a Blaschke product for the right half-plane), and $g_o$ is outer (so all its zeros are in the closed left half-plane or at $\infty$). Then we can write $g_o=g_1g_2$ where $g_1$ is invertible in $L^\infty(i\mathbb{R})$ and $g_2$ has zeros in $i\mathbb{R}\cup \{\infty\}.$ So every rational function $g\in H^2(\mathbb{C}_+)$ can be represented by $g=G_1G_2$ where $G_1$ is invertible in $L^\infty(i \mathbb{R})$ and $G_2$ only has zeros in $i\mathbb{R}\cup \{\infty\}.$ Besides, we can always make the denominator of $G_2$ equal to a power of $(s+1)$ and there will always be at least $1$ as the function is in $H^2(\mathbb{C}_+).$
Now suppose the degrees of the numerator and denominator of $G_2$ are $m$ and $n$, respectively. This means $m$ is the number of imaginary axis zeros of $g$ and $g$ is asymptotic to $s^{m-n}$ at $\infty$ so $n>m.$ In particular, we write $G_2(s)=\prod_{k=1}^m(s-y_k)/(s+1)^n$ with all $y_k\in i\mathbb{R}$. So it yields that \begin{eqnarray*}V^{-1}(g)&=&\frac{2\sqrt{\pi}}{1+z} G_1(M(z))G_2(M(z))\nonumber\\&=&\sqrt{\pi}G_1\left(\frac{1-z}{1+z}\right)
\prod_{k=1}^m \left(\frac{1-z}{1+z}-y_k\right)\left(\frac{1+z}{2}\right)^{n-1}
\nonumber\\&=& 2^{1-n}  \sqrt{\pi}G_1\left(\frac{1-z}{1+z}\right)\prod_{k=1}^m \left(1-y_k-z(1+y_k)\right) (1+z)^{n-m-1},
\end{eqnarray*} with $G_1\left(\frac{1-z}{1+z}\right)$ is rational and invertible in $L^\infty(\mathbb{T})$ and the polynomial $$\prod_{k=1}^m \left(1-y_k-z(1+y_k)\right) (1+z)^{n-m-1}$$ has $n-1$ zeros on $\mathbb{T}$. Combining
this with  Theorem \ref{thm generalg}, we formulate
\begin{equation*}c(V^{-1} g)+ G_1\left(\frac{1-z}{1+z}\right)\phi^\delta K_{z^{n-1}}= G_1\left(\frac{1-z}{1+z}\right) K_{z^n\phi^\delta}. \end{equation*}  And switching into $H^2(\mathbb{C}_+)$ by the map $V$ in \eqref{Vmap}, we obtain the following theorem in $H^2(\mathbb{C}_+)$.
\begin{thm} \label{thm C+} Let $g\in H^2(\mathbb{C}_+)$ be rational with $m$ zeros on the imaginary axis and let $n>m$ such that $s^{n-m} g(s)$ tends to a finite nonzero limit at $\infty.$  Then $g$ can be written as $g=G_1G_2$, where $G_1$ is rational and invertible in $L^\infty(i\mathbb{R})$ and $G_2(s)=\prod_{k=1}^m(s-y_k)/(s+1)^n$ with all $y_k\in i\mathbb{R}$.  Then it holds that
\begin{eqnarray*}\bigvee\{g e^{-\lambda s},\;0\leq \lambda \leq \delta\}+G_1 e^{-\delta s} K_{(\frac{1-s}{1+s})^{n-1}}=G_1
K_{(\frac{1-s}{1+s})^{n}e^{-\delta s}}.\quad \end{eqnarray*}\end{thm}

\begin{rem} Letting  $g(s)=G_1(s)/(1+s)^{n+1}$ in $H^2(\mathbb{C}_+),$ with a  rational and invertible $G_1\in L^\infty(i\mathbb{R}),$ it holds that
$$\bigvee\{G_1\frac{e^{-\lambda s}}{(1+s)^{n+1}}:\;0\leq \lambda\leq \delta\}+G_1 e^{-\delta s} K_{\left(\frac{1-s}{1+s}
\right)^{n}}= G_1K_{\left(\frac{1-s}{1+s}
\right)^{n+1}e^{-\delta s}}.$$ Particularly, for $G_1(s)=1,$  the result $(2)$ of Theorem \ref{thm n} implies $\bigvee\{\frac{e^{-\lambda s}}{(1+s)^{n+1}}:\;0\leq \lambda\leq \delta\}$ has codimension $0$ in $K_{\left(\frac{1-s}{1+s}
\right)^{n+1}e^{-\delta s}}.$ \end{rem}

\begin{rem} For $\widetilde{p}_N$ in Proposition \ref{prop poly}, it also follows that \begin{eqnarray}  \overline{c(\widetilde{p}_N)+ K_{z^N\phi^\delta}}  =K_{z^{N+1}\phi^\delta}.\quad \label{cpnK}\end{eqnarray} Meanwhile, for $g(s)=\prod_{k=1}^m(s-y_k)/(s+1)^n\in H^2(\mathbb{C}_+)$ with all $y_k\in i\mathbb{R}$, the equation \eqref{cpnK} implies that
\begin{eqnarray*}\overline{\bigvee\{g e^{-\lambda s},\;0\leq \lambda \leq \delta\}+ K_{(\frac{1-s}{1+s})^{n-1}e^{-\delta s}}}=
K_{(\frac{1-s}{1+s})^{n}e^{-\delta s}}.\quad \end{eqnarray*}\end{rem}
\begin{proof}Taking the orthogonal complement of  \eqref{pN}  in $H^2(\mathbb{D})$, we obtain \begin{eqnarray*}&&(c(\widetilde{p}_N)+ \phi^\delta K_{z^N})^\bot=(c(\widetilde{p}_N))^\bot \bigcap (\phi^\delta K_{z^N})^\bot\\&&=(c(\widetilde{p}_N))^\bot \bigcap (z^N\phi^\delta H^2\oplus K_{\phi^\delta})
\\&&=z^{N+1}\phi^\delta H^2.\end{eqnarray*}
Here we use \cite[Lemma 2.3]{CGP} to obtain  $(\phi^\delta K_{z^N})^\bot= z^N\phi^\delta H^2\oplus K_{\phi^\delta}.$ Since $ K_{\phi^\delta}$ is orthogonal to $z^N\phi^\delta H^2$ and
$z^N\phi^\delta H^2 \supseteq z^{N+1}\phi^\delta H^2$, the above formula implies
\begin{eqnarray*} (c(\widetilde{p}_N))^\bot \bigcap z^N\phi^\delta H^2=z^{N+1}\phi^\delta H^2.\end{eqnarray*}Then the desired equation \eqref{cpnK} can be deduced by taking the orthogonal complement of the above display  in $H^2(\DD)$.  This presents a link between the model spaces $K_{z^{N+1}\phi^\delta}$ and $K_{z^{N}\phi^\delta}$
in this context.
\end{proof}

\subsection*{Acknowledgments.}
 This work was done while Yuxia Liang was visiting the University of Leeds. She is grateful to the School of Mathematics at the University of Leeds for its warm hospitality. Yuxia Liang is supported  by the National Natural Science Foundation of China (Grant No. 11701422).

\end{document}